\documentclass{amsproc}

 \newtheorem{theorem}{Theorem}[section]
 
 \newtheorem{lemma}[theorem]{Lemma}
 
 \newtheorem{proposition}[theorem]{Proposition}
 \theoremstyle{definition}
 \newtheorem{definition}[theorem]{Definition}
 \theoremstyle{remark}
 \newtheorem{remark}[theorem]{Remark}
 \newtheorem*{example}{Example}

\numberwithin{equation}{section}

\usepackage{color}
\usepackage{graphics}
\usepackage{graphicx}
\usepackage{epsfig}
\usepackage{amssymb}
\usepackage{verbatim}
\usepackage{pb-diagram}
\usepackage{graphics,amsmath,amsfonts,amssymb} 
\usepackage{epsfig,verbatim,multicol}

\begin{document}
\title{Pasting and Reversing Approach to Matrix Theory}
\author[P. Acosta-Hum\'anez]{Primitivo B. Acosta-Hum\'anez}

\address[P. Acosta-Hum\'anez]{School of Basic and Biomedical Sciences - Universidad Sim\'on Bol\'{\i}var, Barranquilla, Colombia}
\email{primitivo.acosta@unisimonbolivar.edu.co -- primi@intelectual.co}
\author[A. Chuquen]{Adriana L. Chuquen}
\address[A. Chuquen]{Master Program in Mathematics, Universidad del Norte, Barranquilla - Colombia}
\email{achuquen@uninorte.edu.co}
\maketitle

\begin{abstract}
The aim of this paper is to study some aspects of matrix theory through Pasting and Reversing. We start giving a summary of previous results concerning to Pasting and Reversing over vectors and matrices, after we rewrite such properties of Pasting and Reversing in matrix theory using linear mappings to finish with new properties and new sets in matrix theory involving Pasting and Reversing. In  particular we introduce new linear mappings: Palindromicing and Antipalindromicing mappings, which allow us to obtain palindromic and antipalindromic vectors and matrices.\\


\end{abstract}

\section*{Introduction}
Pasting and Reversing as mathematical operations, introduced by the first author in \cite{Ac2} and after in \cite{Ac1}, have been studied in different contexts such as natural numbers, polynomials, matrices, vectors, differential operators, permutations and generalized vector product. In the papers \cite{acchro1,acchro2,acmoro} the authors studied these operations in the framework of integer numbers, rings and vector spaces. Following this approach the first author generalized some results, concerning to the product of matrices, of \cite{acchro2} in \cite{acarnu}. Finally, some applications of combinatorial dynamics were done using Pasting and Reversing over simple permutations to obtain genealogies of continuous maps with minimal dynamics, see \cite{Ac1,AM, AM1,AM2}.\\

In this paper, devoted to matrix theory, we follow the same philosophy of the previous ones concerning to Pasting and Reversing. We obtain new results in the framework of elementary matrix theory which arise from links with Pasting and Reversing. In particular we study Pasting and Reversing from linear mappings, proving some interesting results in matrix theory and introducing new linear mappings such as Palindromicing and Antipalindromicing mappings for vector and matrices, among others. For a complete theoretical background in matrix theory see \cite{Br,Ev,Ga,lantis,Le, We}.\\

The paper is organized as follows: Section 1 contains a summary of the properties of Pasting and Reversing related with basic linear algebra following the references \cite{acarnu,acchro1,acchro2}. Section 2 contains the main results of the paper, which in particular allow us the rewriting of some results presented in Section 1, as well to obtain some new results related with Pasting, Reversing and analytic functions over matrices. 

\section{Preliminaries}
We start considering the vector space $V=\mathbb{K}^n$, where $\mathbb{K}$ is a field of characteristic zero, and we write $W\leq V$ to say that $W$ is a subspace of $V$. In this way, $\mathcal M_{n\times m}(\mathbb{K})$ denotes the set of $n\times m$ matrices with elements belonging to $\mathbb{K}$. We should keep in mind that when we write $\mathcal M_{n\times m}$ we means $\mathcal M_{n\times m}(\mathbb{K})$. Finally, \emph{floor} and \emph{ceiling}  functions, denoted by $\lfloor\,\,\rfloor$ and $\lceil\,\,\rceil$ respectively, are defined as
$$\lfloor x\rfloor =\max\{m\in \mathbb {Z} \mid m\leq x\},\quad 
 \lceil x\rceil =\min\{n\in \mathbb {Z} \mid n\geq x\}.$$

The following definitions and properties can be found in \cite{acarnu,acchro1,acchro2}.

Suppose $v=(v_{1},v_{2},\ldots,v_{n}) \in \mathbb{K}^n$, \emph{Reversing} of $v$, denoted by $\widetilde v$ is $$\widetilde{v}=(v_{n},v_{n-1},\ldots,v_{1}).$$ Thus, the following statements hold:
$$ \begin{array}{llll}
1.&\tilde{\tilde{v}}=v&2.&\widetilde{av+bw}=a\widetilde{v}+b\widetilde{w},\,\,a,b\in
\mathbb{K},\,v,w\in V\\
3.&v\cdot w=\widetilde{v}\cdot\widetilde{w}&4.&\widetilde{(v\times w)}=\widetilde{w}\times\widetilde{v},\,\,\forall v,w\in \mathbb{K}^{3}
\end{array}$$

Whenever $\widetilde{v}=v,$ we say that $v$ is a palindromic vector. Similarly, $\widetilde{w}=-w$ means that $w$ is an antipalindromic vector. We denote by $W_p\subset V$ and $W_a\subset V$ the sets of palindromic and antipalindromic vectors respectively, therefore we have the following statements:
$$ \begin{array}{llll}
5.& v+w\in W_p,\,\forall v,w\in W_p & 6.&v+w\in W_a,\,\forall v,w\in W_a\\
7.&v\times w\in W_a,\,\forall v,w\in W_p&8.&v\times w=(0,0,0),\,\forall v,w\in W_a\\
9.&v\times w\in W_p,\,\forall v,\in W_a, \, w\in W_p&10.&v\times w\in W_p,\,\forall v\in W_p, \, w\in W_a
\end{array}$$

We recall that for $n=3$, statements 3 to 6 were proven in \cite{acchro2}, while in \cite{acarnu} were proven for the general case.\\

Consider $v=(v_1,\ldots,v_n)$ and $w=(w_1,\ldots,w_n)$, then Pasting of $v$ with $w$, denoted by $v\diamond w$, is $$v\diamond w=(v_{1},\ldots,v_{n},w_{1},\ldots,w_{m}),$$ therefore $V\diamond W:=\{v\diamond w:\, v\in V, w\in W\}$ and the following
statements hold.
$$\begin{array}{llll}
11.&V \diamond W\cong \mathbb{K}^{n+m}&12.&\dim(V\diamond W) =\dim V+\dim W\\
13.& \widetilde{v\diamond w}=\tilde{w}\diamond \tilde{v}&14.&(v\diamond w)\diamond z=v\diamond (w\diamond z)\\
15.&W_{p}\leq V&16.&\dim W_{p}=\lceil\frac{n}{2}\rceil\\
17.&W_{a} \leq V &18.&\dim W_{a}=\lfloor\frac{n}{2}\rfloor\\
19.&V=W_p\oplus W_a&20.&v=w_p+w_a,\forall v\in V
\end{array}
$$
Due to $(\mathbb{K}_n[x],+,\cdot)\cong (\mathbb{K}^{n+1},+,\cdot)$, we apply Pasting and Reversing over the polynomials to recover some results given in \cite{acchro1,MaRa}.\\

Considering $(\mathcal M_{n\times m},+,\cdot)\cong (\mathbb{K}^{nm},+,\cdot)$, we present two different approaches for Pasting and Reversing to matrices. The first ones is concerning to rows and columns.

Assume $A\in\mathcal M_{n\times m}$, Reversing by rows of $A$ and Reversing by columns of $A$ are given by
\begin{displaymath}
\widetilde{A}_r=\begin{pmatrix}\widetilde{v_{1}}\\ \widetilde{v_{2}}\\ \vdots\\ \widetilde{v_{n}}\end{pmatrix},\quad\widetilde{A}_c=\begin{pmatrix}\widetilde{c_{1}}& \widetilde{c_{2}}& \cdots&\widetilde{c_{m}}\end{pmatrix}.\end{displaymath}
Now, we can assume $A\in \mathcal{M}_{n\times m}(\mathbb{K})$, $B\in\mathcal M_{q\times m}(\mathbb{K})$ and $C\in\mathcal M_{n\times p}(\mathbb{K})$ given as follows. Pasting by rows of $A$ with $C$ and Pasting by columns of $A$ with $B$ are given by
 \begin{displaymath} A\diamond_r C=\begin{pmatrix}z_{1}\\z_{2}\\ \vdots\\z_{n}\end{pmatrix},\quad z_i=v_i\diamond w_i,\quad  A\diamond_c B=\begin{pmatrix}y_{1}&y_{2}& \cdots&y_{n}\end{pmatrix},\quad y_i^T=f_i^T\diamond g_i^T,\end{displaymath}
where $v$ and $w$ are row vectors of $A$ and $B$ respectively, while $f$ and $g$ are column vectors of $C$ and $D$ respectively. Thus the following statements hold.
$$\begin{array}{llll}
21.&\widetilde{\widetilde{A}}_r=A&22.&\widetilde{\widetilde{ A}}_c=A\\
23.&\widetilde{(A\diamond_r B)}_r=(\widetilde {B}_r)\diamond_r (\widetilde{A}_r)&24.&\widetilde{(A\diamond_c B)}_c=(\widetilde{B}_c)\diamond_c (\widetilde{A}_c)\\
25.&(A\diamond_r B)\diamond_r C=A\diamond_r (B\diamond_r C)&26.&(A\diamond_c B)\diamond_c C=A\diamond_c (B\diamond_c C)\\
27.&\widetilde{(\alpha A+\beta B)}_r=\alpha\widetilde{ A}_r+\beta \widetilde{B}_r&28.&\widetilde{(\alpha A+\beta B)}_c=\alpha \widetilde{A}_c+\beta\widetilde{B}_c\\
29.&\mathcal M_{n\times m}\diamond_r \mathcal M_{n\times p}= \mathcal M_{n\times(m+p)}&30.&\mathcal M_{n\times m} \diamond_c \mathcal M_{l\times m}= \mathcal M_{(n+l)\times m}\\
\end{array}$$
The following properties can allow us the writing of Pasting in an algorithmic way
$$\begin{array}{ll}
31.&A\diamond_rB=A((I_n\diamond_c\mathbf{0}_{(n-m)\times m})\diamond_r \mathbf{0}_{n\times p})+\mathbf{0}_{n\times m}\diamond_r B\\
32.&A\diamond_cB=A((I_n\diamond_r\mathbf{0}_{n\times (m-q)})\diamond_c \mathbf{0}_{n\times p}))+\mathbf{0}_{n\times m}\diamond_c B\\
\end{array}$$

We denote by $W_p^r(n\times m)\subset \mathcal M_{n\times m}$, $W_p^c(n\times m)\subset \mathcal M_{n\times m}$, $W_a^r(n\times m)\subset\mathcal M_{n\times m}$ and $W_a^c(n\times m)\subset\mathcal M_{n\times m}$ the sets of palindromic matrices by rows, palindromic matrices by columns, antipalindromic matrices by rows and antipalindromic matrices by columns respectively. In some cases, for pedagogical purposes, we denote them by $W_p^r$, $W_p^c$, $W_a^r$ and $W_a^c$. Therefore we have the following statements:
$$\begin{array}{llll}
33.&W^r_p\leq\mathcal M_{n\times m}&34.&W^c_p\leq\mathcal M_{n\times m}\\
35.&\dim W^r_p=n\left\lceil \frac{m}{2}\right\rceil&36.&\dim W^c_p=m\left\lceil \frac{n}{2}\right\rceil\\
37.&W^r_a\leq\mathcal M_{n\times m}&38.&W^c_a\leq\mathcal M_{n\times m}\\
39.&\dim W^r_a=n\left\lfloor \frac{m}{2}\right\rfloor&40.&\dim W^c_a=m\left\lfloor \frac{n}{2}\right\rfloor\\
41.&\mathcal M_{n\times m}=W^c_p\oplus W^c_a&42.&\mathcal M_{n\times m}=W^r_p\oplus W^r_a\\
\end{array}$$

Some properties derived from Pasting and Reversing by rows and columns with respect to classical matrices operations are:
$$\begin{array}{llll}
43.&(\widetilde{A}_r)^T=\widetilde{(A^T)}_c&44.&(\widetilde{ A}_c)^T=\widetilde{(A^T)}_r\\
45.&(A\diamond_c B)^T=A^T\diamond_r B^T&46.&(A\diamond_r B)^T=A^T\diamond_c B^T\\
47.&\widetilde{(AB)}_r=A(\widetilde{B}_r)&48.&\widetilde{(AB)}_c=(\widetilde{A}_c)B\\
49.&\det(\widetilde{A}_c)=(-1)^{\left\lfloor\frac{n}{2}\right\rfloor}\det A&50.&\det(\widetilde{A}_r)=(-1)^{\left\lfloor\frac{n}{2}\right\rfloor}\det A\\
51.&(\widetilde{A}_c)^{-1}=\widetilde{(A^{-1})}_r&52.&(\widetilde{A}_r)^{-1}=\widetilde{(A^{-1})}_c\\
53.&A=A^r_p+A^r_a,\,\forall A\in\mathcal M_{n\times m} &54.&A=A^c_p+A^c_a,\,\forall A\in\mathcal M_{n\times m}\\
\end{array}$$
Some properties derived from palindromic and antipalindromic matrices, with respect to classical matrix theory, are:
$$\begin{array}{llll}
55.&A+B\in W^r_p,\,\forall A,B\in W^r_p&56.&A+B\in W^r_a,\,\forall A,B\in W^r_a\\
57.&A+B\in W^c_p,\,\forall A,B\in W^c_p&58.&A+B\in W^c_a,\,\forall A,B\in W^c_a\\
59.&\forall A,B\in W^r_p,\,AB\in W^r_p&60.&\forall A,B\in W^r_a,\,AB\in W^r_a\\
61.&\forall A,B\in W^c_p,\,AB\in W^c_p&62.&\forall A,B\in W^c_a,\,AB\in W^c_a\\
63.&AB\in W^r_p\Leftrightarrow B\in W^r_p&64.&AB\in W^r_a\Leftrightarrow B\in W^r_a\\
65.&AB\in W^c_p\Leftrightarrow A\in W^c_p&66.&AB\in W^c_a\Leftrightarrow A\in W^c_a\\
67.&AB\neq \mathbf{0}\in W^r_p\Leftrightarrow B\in W^r_p&68.&AB\neq \mathbf{0}\in W^c_p\Leftrightarrow A\in W^c_p\\
69.&AB\neq \mathbf{0}\in W^r_a\Leftrightarrow B\in W^r_a&70.&AB\neq \mathbf{0}\in W^c_a\Leftrightarrow A\in W^c_a\\
\end{array}$$
Now, introducing the sets \emph{double palindromic matrices} $\mathsf{W}_{pp}:=W_p^r\cap W_p^c$, \emph{double antipalindromic matrices} $\mathsf{W}_{aa}:=W_a^r\cap W_a^c$ and $\mathsf{W}_{pa}:=W_p^r\cap W_a^c,\, \mathsf{W}_{ap}:=W_a^r\cap W_p^c,$ therefore the following statements hold.
$$\begin{array}{llll}
71.&\mathsf{W}_{pp}\leq\mathcal M_{n\times m}& 72.&\mathsf{W}_{pa}\leq\mathcal M_{n\times m}\\
73.&\mathsf{W}_{ap}\leq\mathcal M_{n\times m}&
74.&\mathsf{W}_{aa}\leq\mathcal M_{n\times m}\\
75.&\dim \mathsf W_{pp}=\left\lceil \frac{n}{2}\right\rceil\left\lceil \frac{m}{2}\right\rceil&
\smallskip

76.&\dim \mathsf W_{pa}=\left\lceil \frac{n}{2}\right\rceil\left\lfloor \frac{m}{2}\right\rfloor\\
77.&\dim \mathsf W_{ap}=\left\lfloor \frac{n}{2}\right\rfloor\left\lceil \frac{m}{2}\right\rceil&
78.&\dim \mathsf W_{aa}=\left\lfloor \frac{n}{2}\right\rfloor\left\lfloor \frac{m}{2}\right\rfloor\\
79.&\mathcal W=\mathsf W_{pp}\oplus \mathsf W_{pa}\oplus\mathsf W_{ap}\oplus \mathsf W_{aa}&
80.&\forall A\in\mathcal W,\,A=A_{pp}+A_{pa}+A_{ap}+A_{aa}
\end{array}$$

In \cite[\S 3]{acarnu} can be found a relationship between a generalized vector product of $n-1$ vectors of $\mathbb{K}^n$ and Reversing. Assume $v_{i}\in\mathbb K^n$, the generalized vector product of $v_i$ is given by
$$\bigwedge_{i=1}^{n}  (v_{i})=\sum_{k=1}^{n}\left(
-1\right) ^{1+k}\det \left( M^{(k)}\right) e_{k},
$$
Where the matrix $M^{\left( k\right) }\in \mathcal M_{(n-1)\times(n-1)}$ and is given by
$$
M^{(k)}=\left\{
\begin{array}{l}
\left(m_{i,j}\right)\text{, whether }j<k \\
\left(m_{i,j+1}\right)\text{, whether }j\geq k%
\end{array}
\right.
$$
Now, consider $M\in\mathcal M_{\left(
n-1\right) \times n}$ and $v_i\in \mathbb K^n$ where $1\leq i \leq n-1$, then we can obtain the following properties
$$\begin{array}{ll}
81.&\widetilde{(M^{\left( k\right) })}_r=M^{\left( n-k+1\right) }\widetilde I_{n-1}, 1\leq k\leq n\\
82.&\displaystyle \bigwedge_{i=1}^{n-1} (\widetilde{v_i}_r)=(-1)^{\left\lceil\frac{3n}{2}\right\rceil}\widetilde{\left(\bigwedge_{i=1}^{n-1} (v_{i})\right)}_r\\
 \end{array}$$

Now, considering the matrices $A\in \mathcal M_{n\times m}$ and $B\in \mathcal M_{p\times q}$ as vectors,  we use $\widehat{A}$ to denote Reversing of $A$, that is  $$\widehat{A}=A\widehat I_{nm}.$$ Also for $n=p$ or $m=q$ (exclusively) we denote by $A\diamond B$ Pasting of $A$ with $B$, which is given by $$A\diamond B=(v_{11},\ldots,v_{1m},\ldots,v_{n1},\ldots v_{nm},w_{11},\ldots,w_{1q},\ldots,w_{p1},\ldots, w_{pq}).$$ Therefore, we return to recover the expressions for $\widehat{A}$ and $A\diamond B$ in term of matrices instead of vectors, i.e.,
$$\widehat{A}=\begin{pmatrix}
v_{nm}&\ldots&v_{n1}\\\vdots\\v_{1m}&\ldots& v_{11}
\end{pmatrix},\, A\diamond B=\left\lbrace\begin{array}{l}
\begin{pmatrix}
v_{11}&\ldots&v_{1m}&w_{11}&\ldots&w_{1q}\\ \vdots\\v_{n1}&\ldots& v_{nm}&w_{p1}&\ldots&w_{pq}
\end{pmatrix},\, \begin{matrix}n=p\\m\neq q\end{matrix} \\ \\\begin{pmatrix}
v_{11}&\ldots&v_{1m}\\ \vdots\\v_{n1}&\ldots& v_{nm}\\w_{11}&\ldots&w_{1p}\\ \vdots\\w_{p1}&\ldots&w_{pq}
\end{pmatrix},\,\begin{matrix}n\neq p\\m= q\end{matrix}\end{array}\right.$$ We say that any matrix $P$, is a palindromic matrix whether $\widehat{P}=P$. In the same way, we say that any matrix $A$, is an antipalindromic matrix whether $\widehat{A}=-A$.
Note that $\widehat{(I_n)}=I_n\widehat{I}_{n^2}$. From now on, the sets of palindromic and antipalindromic matrices are denoted by $\mathbf{PA}(n\times m,\mathbb{K})$ and $\mathbf{aPA}(n\times m,\mathbb{K})$ respectively. Moreover, we write $\mathbf{PA}(n,\mathbb{K})$ and $\mathbf{aPA}(n,\mathbb{K})$ whether $n=m$. In some cases, for pedagogical purposes, we write the Set of Palindromic Matrices as $\mathbf{PA}(n\times m)$ and $\mathbf{PA}$. Similarly, also we write the Set of Antipalindromic Matrices as $\mathbf{aPA}(n\times m$ and $\mathbf{aPA}$. Thus, we arrive to the following elementary results for $A \in \mathcal M_{n\times m}$, $B \in \mathcal M_{p\times q}$, $C \in \mathcal M_{r\times s}$, $b,c\in \mathbb{K}$, $A_a\in \mathbf{aPA}$ and $A_p\in \mathbf{PA}$.
$$\begin{array}{llll}
83.&\widehat{A}=(\widetilde{v_n},\cdots,\widetilde{v_1})&84.&\widehat{\widehat{A}}=A\\
85.&\widehat{(A\diamond B)}=\widehat{B}\diamond \widehat{A}&86.&(A\diamond B)\diamond C=A\diamond (B\diamond C)\\
87.&p=n,\, q=m,\,\widehat{(bA+cB)}=b\widehat{A}+c\widehat{B}&88.&A\diamond B=\mathcal M_{r\times s}(\mathbb{K})\\
89.&\mathbf{PA}\leq\mathcal M_{n\times m}&90.&\mathbf{aPA}\leq\mathcal M_{n\times m}\\
91.&\dim \mathbf{PA}=\left\lceil \frac{nm}{2}\right\rceil&92.&\dim \mathbf{aPA}=\left\lfloor \frac{nm}{2}\right\rfloor\\
93.&A+B\in \mathbf{PA},\,A,B\in \mathbf{PA}&94.&A+B\in \mathbf{aPA},\,A,B\in \mathbf{aPA}\\
95.&\forall A\in\mathcal M_{n\times m},A=A_p+A_a&96.& \mathcal M_{n\times m}=\mathbf{PA}\oplus \mathbf{PA}.
\end{array}$$

Also, we can see some links between classical matrix theory and Reversing in the following results.

$$\begin{array}{llll}
97.&\widehat{I_n}=I_n&98.& \widehat{A}=\widetilde{\left(\widetilde{A}_r\right)}_c\\ 99.&\widehat{A}=\widetilde{\left(\widetilde{A}_c\right)}_r&100.&\widehat{(AB)}=\widehat{(A)}\widehat{(B)}\\
101.&(\widehat{A})^{-1}=\widehat{(A^{-1})}&102.&\det (\widehat{A})=\det A\\
103.&\mathrm{Tr}(\widehat{A})=\mathrm{Tr} A&104.&\widehat{A^T}=(\widehat{A})^T\\
105.&AB\in W_p,A,B\in W_p&106.&AB\in W_p,A,B\in W_a\\
107.&AB\in W_a, A\in W_pB\in W_a\\
\end{array}$$
Now, \emph{Pasting by blocks} is given by:
$$A\diamond_b B:=\begin{pmatrix}A & \mathbf{0}_{n\times s}\\\mathbf{0}_{r\times m} &B \end{pmatrix}\in\mathcal{M}_{(n+r)\times(m+s)}.$$
where $A\in\mathcal{M}_{n\times m}$ and $B\in\mathcal{M}_{r\times s}.$ It is well known that Pasting by blocks corresponds to a particular case of \emph{block matrices}, also called \emph{partitioned matrices}, see \cite{hall,lantis}. The next result,  is consequence of properties $84$ to $106$. \\

Consider the matrices $A \in M_{n\times m}$, $B \in M_{p\times q}$ and $C \in M_{r\times s}$, we have:
$$\begin{array}{llll}
108.&\widehat{(A\diamond_b B)}=\widehat{(B)}\diamond_b \widehat{(A)}&109.&(A\diamond_b B)\diamond_b C=A\diamond_b (B\diamond_b C)\\
110.&A \diamond_b B\in \mathcal M_{(n+p)\times (m+q)}&111.&(A\diamond_b B)^T=A^T\diamond_b B^T\\
112.&\det (A\diamond_b B)=\det A  \det B&113.&\mathrm{Tr}(A\diamond_b B)=\mathrm{Tr} A+\mathrm{Tr} B\\
114.&(A\diamond_b B)^{-1}=A^{-1}\diamond_b B^{-1}& &\\
\end{array}$$
\bigskip

\section{Main Results}
This section contains original results, which can be implemented in courses of linear algebra.

We denote by $ R$ Reversing mapping, i.e., $\mathcal R:V\rightarrow V$ where for all $v=(v_1,v_2,\ldots,v_n)$, $\mathcal R(v)=(v_n,\ldots,v_2,v_1)$. Thus, we obtain the following results.
\begin{proposition}\label{proplint}
The following statements hold
\begin{enumerate}
\item  $\mathcal R\in Aut(V)$.
\item The transformation matrix of $\mathcal R$ is given by  \begin{displaymath} M_\mathcal {R}=(\delta_{i,n-j+1})_{n\times n}.\end{displaymath}
\item Minimal and characteristic polynomial of $\mathcal{R}$ are given respectively by $$Q(\lambda)=\lambda^2-1,\quad P(\lambda)=P(\lambda)=(\lambda+1)^{\lfloor \frac{n}{2}\rfloor}(\lambda-1)^{\lceil\frac{n}{2}\rceil}.$$
\item $\ker(\mathcal R- id)^{\perp}=\ker(\mathcal R+id) $, $id(v)=v$, $\forall v\in V$.
\item $V=\ker(\mathcal R-id)\oplus \ker(\mathcal R+id)$.
\item $\dim \ker(\mathcal R-id)=\left\lceil\frac{n}{2}\right\rceil$, $\dim \ker(\mathcal R+id)=\left\lfloor\frac{n}{2}\right\rfloor$.
 \end{enumerate}
\end{proposition}
\begin{proof} We proceed according to each item.
\begin{enumerate}
\item For any scalar $\alpha\in K$ and any pair of vectors $v,w\in V$ we obtain
$\mathcal R(v+w)=\widetilde{v+w}=\widetilde{v}+\widetilde{w}=\mathcal{R}v+\mathcal{R}w$ and $\mathcal R(\alpha v)=\widetilde{\alpha v}=\alpha\widetilde{v}=\alpha\mathcal{R}v$. Owing to $\mathcal{R}:V\rightarrow V$ and $\widetilde{\widetilde v}=v$ implies that $\mathcal R$ is left-right invertible, then $\mathcal R$ is an automorphism of $V$.
\item By definition of Reversing we have
$\mathcal{R}:V\rightarrow V$ is such that $$\mathcal R(v_1,v_2,\ldots, v_{n-1},v_n)=(v_1,v_2,\ldots, v_{n-1},v_n)(\delta_{i,n-j+1})_{n\times n}.$$

\item We observe that $\mathcal R^2=id$ and $\mathcal R\neq id$, therefore $Q(\lambda)=\lambda^2-1$ is the minimal polynomial of $\mathcal{R}$, that is, $Q(\delta_{i,n-j+1})=\mathbf 0_n\in\mathcal M_{n\times m}$ ($\mathbf 0_n$ is the zero matrix of size $n\times n$). For instance, $Q$ divides to the characteristic polynomial of  $\mathcal R$. Now, assuming $n=2$ we have that the characteristic polynomial of $\mathcal R$ is given by  $P(\lambda)=(\lambda+1)(\lambda-1)$, assuming $n=3$ we obtain $P(\lambda)=(\lambda+1)(\lambda-1)^2$. Therefore the characteristic polynomial of $\mathcal R$ is obtained inductively and it is given by
\begin{displaymath}P(\lambda)=\det (\delta_{i,n-j+1}-\lambda I_n)=\displaystyle{ \left\{ { (\lambda+1)^m(\lambda-1)^m\,for\,n=2m \atop (\lambda+1)^m(\lambda-1)^{m+1}\,for\,n=2m+1         } \right.      }.\end{displaymath} That is, by definition of floor and ceiling functions we conclude $$P(\lambda)=(\lambda+1)^{\lfloor \frac{n}{2}\rfloor}(\lambda-1)^{\lceil\frac{n}{2}\rceil}.$$ 
\item Assume $v\in \ker(\mathcal R- id)$ and $w\in \ker(\mathcal R+ id)$, for instance $\mathcal Rv=v$, $\mathcal R(w)=-w$ and $v\cdot w=-v\cdot w=0$. In this way $\ker(\mathcal R- id)^{\perp}=\ker(\mathcal R+id) $.
\item Due to $\mathcal R^2-id=(\mathcal R-id)(\mathcal R+id)$, then $V=\ker(\mathcal R-id)\oplus \ker(\mathcal R+id)$.
\item  By definition of $\lfloor\frac{n}{2}\rfloor$,  $\lceil\frac{n}{2}\rceil$ and item 3 we have that $(\lambda-1)^{\lceil\frac{n}{2}\rceil}(\lambda+1)^{\lfloor\frac{n}{2}\rfloor}=P(\lambda)$. Now, by item 5 we have $V=\ker(\mathcal R-id)\oplus \ker(\mathcal R+id)$, then we conclude $\dim \ker(\mathcal R-id)=\lceil\frac{n}{2}\rceil$, $\dim \ker(\mathcal R+id)=\lfloor\frac{n}{2}\rfloor$ and $\lceil\frac{n}{2}\rceil+\lfloor\frac{n}{2}\rfloor=n$.
\end{enumerate}
\end{proof}
\begin{remark} Proposition \ref{proplint} summarizes some results given in \cite[\S 1]{acchro2}, without the formalism of endomorphism. In particular, $W_p:= \ker(\mathcal R-id),$ $W_a:=\ker(\mathcal R+id)$ and $\mathcal{R}$ is an endomorphism associated to a \emph{permutation matrix}. Recall that $A_{\sigma}$ is a permutation matrix, defined over a given $\sigma\in S_n$, whether its associated linear mapping $\mathcal{R}_{\sigma}$ is given by $$\mathcal{R}_{\sigma}\,:\begin{array}{ccc}
 V&\longrightarrow& V\\(v_1,\ldots,v_n)&\mapsto&(v_{\sigma(1)},\ldots,v_{\sigma(n)}).
 \end{array}
 $$
 This means that Reversing corresponds to $\mathcal R_\sigma=\mathcal R$, where the permutation $\sigma$ and the matrix $A_\sigma$ are given respectively by $$\sigma=\begin{pmatrix}
 1&2&3&\ldots&n-1&n\\n&n-1&n-2&\ldots&2&1
 \end{pmatrix}, \quad A_\sigma=\begin{pmatrix}
 0&0&0&\ldots&0&1\\0&0&0&\ldots&1&0\\\vdots\\0&1&0&\ldots&0&0\\1&0&0&\ldots&0&0
 \end{pmatrix}.$$
 \end{remark}
To illustrate this formalism, we rewrite items $1,2,3$ and $4$.

$$\begin{array}{ll}
1.&\mathcal R^2(v)=v\\
2.&\mathcal R(av+bw)=a\mathcal R(v)+b\mathcal R(w),\,\,a,b\in
K,\,v,w\in V\\
3.&v\cdot w=\mathcal R(v)\cdot\mathcal R(w)\\
4.&\mathcal R(v\times w)=\mathcal R(w)\times\mathcal R(v),\,\,\forall v,w\in K^{3}\end{array}$$

We say that a mapping is \textit{palindromicing} whether it transform any vector of a given vector space into a palindromic vector, i.e., it is epimorphism from $V$ to $\ker(\mathcal R-id)$. In a similar way, we say that a mapping is \textit{antipalindromicing} whether it transform any vector of a given vector space into an antipalindromic vector, i.e., it is epimorphism from $V$ to $\ker(\mathcal R+id)$. Moreover, we say that palindromicing and antipalindromicing mappings are \textit{canonical},  denoted by $\mathcal F_p$ and $\mathcal F_a$ respectively, whether they are linear mappings and for all $v\in V$ they satisfy $v=(\mathcal F_p+\mathcal F_a)(v)$ and $\mathcal R(v)=(\mathcal F_p-\mathcal F_a)(v)$. From now on we only consider canonical palindromicing and antipalindromicing mappings, which will be called \textit{Palindromicing} and \textit{Antipalindromicing} mappings. The following proposition give us the characterization of Palindromicing and Antipalindromicing mappings.
\begin{proposition}[Palindromicing and Antipalindromicing mappings]\label{proppalant}
The following statements hold.
\begin{enumerate}
\item $\mathcal F_p$ and $\mathcal F_a$ are given by
\begin{displaymath}\begin{array}{lllllllllllll}
 \mathcal F_p&:& &V&\rightarrow&\ker(\mathcal R-id)&&& \mathcal F_a&:& V&\rightarrow&\ker(\mathcal R+id) \\ &&&&&&,&&&&&&\\&&& v&\mapsto&\frac12(v+\mathcal R(v))&&&&& v&\mapsto&\frac12(v-\mathcal R(v))\end{array}.
 \end{displaymath} 
\item $\ker(\mathcal F_p)=\mathrm{Im}(\mathcal{F}_a)$, $\ker(\mathcal F_a)=\mathrm{Im}(\mathcal{F}_p)$.
\item The companion matrices of $\mathcal F_p$ and $\mathcal F_a$ are $M_\mathcal{R}+I_n$ and $M_\mathcal{R}-I_n$.
\end{enumerate}
\end{proposition}
\begin{proof}
Consider $v=(v_1,\ldots,v_n)$. We proceed according to each item
\begin{enumerate}
\item We can see that $\mathcal F_p$ and $\mathcal F_a$ are linear mappings due to $\mathcal R$ and $id$ are linear mappings. Now, $\mathcal{F}_p(v)=\frac{1}{2}(v_1+v_n,v_2+v_{n-1},\ldots, v_{n-1}+v_2,v_n+v_1)$ is a palindromic vector, $\mathcal{F}_a(v)=\frac{1}{2}(v_1-v_n,v_2-v_{n-1},\ldots, v_{n-1}-v_2,v_n-v_1)$ is an antipalindromic vector, for instance $\mathcal F_p$ and $\mathcal F_a$ are epimorphisms from $V$ to $\ker(\mathcal R-id)$ and from $V$ to $\ker(\mathcal R-id)$. Furthermore, we see that $(\mathcal F_p+\mathcal F_a)(v)=\mathcal{F}_p(v)+\mathcal{F}_a(v)=v$ and $(\mathcal F_p-\mathcal F_a)(v)=\mathcal{F}_p(v)-\mathcal{F}_a(v)=\mathcal R(v)$. Thus, $\mathcal F_p$ is the Palindromicing mapping and $\mathcal F_a$ is the Antipalindromicing mapping.
\item Assume $v\in\ker\mathcal{F}_p$ and $w\in\ker\mathcal{F}_a$. We see that $\mathcal{F}_p(v)=0$ implies that $\mathcal{R}(v)=-v$, thus $v\in W_a=\ker(\mathcal R+id)=\mathrm{Im}(\mathcal{F}_a)$. Similarly, we see that $\mathcal{F}_a(w)=0$ implies that $\mathcal{R}(w)=w$, thus $w\in W_p=\ker(\mathcal R-id)=\mathrm{Im}(\mathcal{F}_p)$.
\item It follows directly from the companion matrices of the linear mappings $\mathcal{R}$ and $id$.
\end{enumerate}
\end{proof}
\begin{remark} Palindromicing and Antipalindromicing mappings are not isomorphisms due to they are not monomorphisms. This fact is proven in item 2 of Proposition \ref{proppalant}.
\end{remark}
Now we define Pasting between vectors in the following form.

\begin{definition} Let $V$, $W$ and $Z$ be $\mathbb K$-vector spaces. Consider $v=(v_1,\ldots,v_n)\in V$ and $w=(w_1,\ldots,w_m)\in W$. Pasting is the mapping $\mathcal P$ such that $$ \mathcal P: \begin{array}{ccc}
V\times W&\rightarrow& Z\\(v,w)&\mapsto&\mathcal P(v,w)=z,\end{array}$$ where $z=(v_1,\ldots,v_n,w_1,\ldots,w_m)$. Furthermore $\mathcal P(V,W):=\{\mathcal P(v,w):\,v\in V, w\in W\}$.
\end{definition}

\begin{theorem}\label{teor1} Consider $V=\mathbb{K}^n$, $W=\mathbb{K}^m$, $Z= \mathbb{K}^{n+m}$, $V'\leq Z$ and $W'\leq Z$ generated by $\{e_1^{n+m},\ldots, e_n^{n+m}\}$ and $\{e_{n+1}^{n+m},\ldots, e_{n+m}^{n+m}\}$ respectively, where $e_i^r$ is the $i$-th vector of the canonical basis of $\mathbb{K}^r$. The mappings $\varphi_1$, $\varphi_2$, $\varphi$ and $\mathcal S$ are given by
\begin{equation*}
\begin{split}\varphi_1:&V\rightarrow V'\\
&v\rightarrow \varphi_1(v)=v',\, v_i=v'_i,\,\,\forall i\leq n
\end{split}\quad \quad 
\begin{split}\varphi_2:&W\rightarrow W'\\
&w\rightarrow \varphi_2(w)=w',\, w_i=w'_{n+i},\,\,\forall i\leq m
\end{split}
\end{equation*}

\begin{equation*}
\begin{split}\varphi:&V\times W\rightarrow V'\times W'\\
&(v,w)\rightarrow\varphi(v,w)=(\varphi_1( v),\varphi_2(w))
\end{split}
\begin{split}\mathcal S:&V'\times W'\rightarrow Z\\
&(v',w')\rightarrow S(v',w')=v'+w'
\end{split}
\end{equation*}
The following statements hold
\begin{enumerate}
\item Mappings $\varphi_1$ and $\varphi_2$  are linear isomorphisms.
\item The transformation matrices of $\varphi_1$ and $\varphi_2$ are given by  $M_{\varphi_1}=(\delta_{i,j})_{n\times (n+m)}$ and $M_{\varphi_2}=(\delta_{i,j-n})_{m\times (n+m)}$ respectively.
\item $\mathrm{im}(\varphi_1)\oplus\mathrm{im}(\varphi_2)=V$.
\item $(\mathrm{im}(\varphi_1))^\perp=\mathrm{im}(\varphi_2)$.
\item The diagram
$$\begin{diagram}
\node{V\times W} \arrow{e,t}{\mathcal P} \arrow{s,l}{\varphi=(\varphi_1,\varphi_2)} \node{Z}\\
\node{V'\times W'}\arrow{ne,r}{S}\end{diagram}$$ is commutative, i.e., $\mathcal P=S\circ\varphi$.
\end{enumerate}
\end{theorem}
\begin{proof} We proceed according to each item:
\begin{enumerate}
\item Assume $\alpha,\beta\in \mathbb{K}$, $v_1,v_2 \in V$ and $w_1,w_2\in W$ such that $v_1=(v_{11},\ldots,v_{1n})$, $v_2=(v_{21},\ldots v_{2n})$, $w_1=(w_{11},\ldots,w_{1n})$, $w_2=(w_{21},\ldots w_{2n})$. By definition of $\varphi_1$ and $\varphi_2$ we have that $$\varphi_1(v_1)=(v_{11},\ldots,v_{1n},\underbrace{0,\ldots,0}_{m\,\, times}),\quad \varphi_1(v_2)=(v_{21},\ldots,v_{2n},\underbrace{0,\ldots,0}_{m\,\, times},$$
$$\varphi_2(w_1)=(\underbrace{0,\ldots,0}_{n\,\, times},w_{11},\ldots,w_{1n}),\quad \varphi_2(w_2)=(\underbrace{0,\ldots,0}_{n\,\, times},w_{21},\ldots,w_{2n}).$$ Therefore, $\varphi_1(\alpha v_1+\beta v_2)=\alpha \varphi_1(v_1)+\beta \varphi_1(v_2)$ and  $\varphi_2(\alpha w_1+\beta w_2)=\alpha \varphi_2(w_1)+\beta \varphi_2(w_2)$. Now, due to $\varphi_1(v)=\mathbf{0}_{V'}$ if and only if $v=\mathbf{0}_V$ and  $\varphi_2(w)=\mathbf{0}_{W'}$ if and only if $w=\mathbf{0}_W$, we get $\ker\varphi_1=\{\mathbf{0}_V\}$ and $\ker\varphi_2=\{\mathbf{0}_W\}$. Finally, for all $v'\in V'$ and for all $w'\in W'$ we get that there exist $v\in V$ and $w\in W$ such that $\varphi_1(v)\in V'$ and   $\varphi_2(w)\in W'$, i.e., $\mathrm{im}(\varphi_1)=V'$ and $\mathrm{im}(\varphi_2)=W'$. Thus, we conclude that $\varphi_1$ and $\varphi_2$ are linear mappings, monomorphisms and epimorphisms.
\item We see that for all $v\in V$ and  for all $w\in W$, we obtain $$\varphi_1(v)=v(\delta_{i,j})_{n\times (n+m)}=vM_{\varphi_1}, \quad \varphi_2(w)=w(\delta_{i,j-n})_{m\times (n+m)}=wM_{\varphi_2}.$$
\item By item 1 we have that $\mathrm{im}(\varphi_1)=V'$ and $\mathrm{im}(\varphi_2)=W'$. We see that $V'\cap W'=\{\mathbf{0}_Z\}$ and $Z=\{\alpha v+\beta w: \alpha, \beta \in \mathbb{K}, v\in V, w\in W$\}.
\item Owing to $v'\cdot w'=0$ for all $v'\in V'$ and $w'\in W'$ and by previous item we have that $W'$ is the orthogonal subspace of $V'$.
\item 
Let $v\in V$ and $w\in W$ then 
\begin{equation*}
\begin{split}(S\circ\varphi)(v,w)&=S(\varphi_1(v),\varphi_2(w))\\
&=\varphi_1(v)+\varphi_2(w)\\
&=\mathcal P(v,w)
\end{split}
\end{equation*}
therefore $\mathcal P=S\circ\varphi$ 
\end{enumerate}
\end{proof}
We illustrate item 5 of previous proposition through the following example.
\begin{example}
Consider the vectors $(1,3,4,5)$, $(2,4,3)$	$(S\circ\varphi)((1,3,4,5),(2,4,3))=S((1,3,4,5,0,0,0),(0,0,0,0,2,4,3))=(1,3,4,5,2,4,3)$
\end{example}

%

To illustrate this formalism we rewrite the following items.
$$\begin{array}{ll}
11.&\mathcal P(V,W)\cong K^{n+m}\\
12.&\dim(\mathcal P(V,W)) =\dim V+\dim W\\
13.&\mathcal R(\mathcal P(v,w))=\mathcal P(\mathcal R(v),\mathcal R(w))\\
14.&\mathcal P(\mathcal P(v,w),z)=\mathcal P(v,\mathcal P(w,z))\end{array}$$

Now we define Reversing and Pasting mappings for matrices. We start introducing Reversing of matrices of three different kinds.
\begin{definition} Assume $A=(a_{ij})_{n\times m}$, $a_{ij}\in \mathbb K$. Reversing for rows of $A$, denoted by $\mathcal R_r(A)$,  Reversing for columns of $A$, denoted by $\mathcal R_c(A)$, and  Reversing of $A$, denoted by $\mathcal R(A)$, are given by $$\mathcal R_r(A)=(a_{i(m+1-j)})_{n\times m},\quad \mathcal R_c(A)=(a_{(n+1-i)j})_{n\times m},\quad \mathcal R(A)=(a_{(n+1-i)(m+1-j)})_{n\times m}$$ respectively.
\end{definition}
We illustrate in the following example each reversing for matrices.
\begin{example}
	Consider $A=\begin{pmatrix}4&6&8&8\\1&3&5&4\\3&2&7&7\end{pmatrix}$
	\begin{enumerate} 
		\item Reversing by rows of $A$ is given by 
		$\mathcal R_r(A)=\begin{pmatrix}8&8&6&4\\4&5&3&1\\7&7&2&3\end{pmatrix} $
		\item Reversing by columns of $A$ is given by
		$\mathcal R_c(A)=\begin{pmatrix}3&2&7&7\\1&3&5&4\\4&6&8&8\end{pmatrix}$
		\item Reversing of $A$ is given by
		$\mathcal R(A)=\begin{pmatrix}7&7&2&3\\4&5&3&1\\8&8&6&4\end{pmatrix}$
	\end{enumerate}
\end{example}

To avoid confusion, assume Reversing of vectors and Reversing of matrices transformations denoted by $\mathcal R_v$ and $\mathcal R_m$ respectively. The ``vectorization'' mapping, denoted by $\rho_{nm}$, transforms a matrix belonging to $\mathcal M_{n\times m}(\mathbb{K})$ in a vector belonging to $\mathbb{K}^{nm}$. We can recover the previous results involving Reversing of vector with Reversing of matrices transformations through the vectorization mapping, which is presented in the following lemma.
\begin{lemma}\label{teor2} Mapping $\rho_{nm}$ is a linear isomorphism between $\mathcal M_{n\times m}$ and $\mathbb K^{nm}$.
\end{lemma}
\begin{proof}
It is well known that $\mathcal M_{n\times m}$ and $\mathbb K^{nm}$ are isomorphic vector spaces. Using basic linear algebra we can see that $\rho_{nm}$ is a linear transformation, is a monomorphism and is an epimorphism.
\end{proof}
\begin{remark}
Using Lemma \ref{teor2} we obtain the analogue version for matrices of Theorem \ref{teor1}.
\end{remark}
Now we introduce the definition of Pasting mappings by rows, columns and blocks.
\begin{definition}
Assume $A\in\mathcal M_{n\times m},\, B\in  \mathcal M_{n\times p}$ and $C\in  \mathcal M_{q\times m}$, Pasting by rows of $A$ with $B$, denoted by $\mathcal P_r(A,B)$ is given by $$\mathcal P_r(A,B)= \begin{pmatrix} (S\circ\varphi)(a_1,b_1)\\\vdots\\(S\circ\varphi)(a_n,b_n) \end{pmatrix},\quad\textrm{where}\quad a_i\in K^m,\,b_i\in K^p.$$ Pasting by columns of $A$ with $C$, denoted by $\mathcal P_c(A,C)$, is given by $$\mathcal P_c(A,C)=\begin{pmatrix} (S\circ\varphi)(a_1,c_1)\ldots(S\circ\varphi)(a_m,c_m) \end{pmatrix},\quad\textrm{where}\quad a_j\in K^n,\,c_j\in K^q.$$ 
Pasting by blocks of $B$ with $C$, denoted by $\mathcal P_b(B,C)$, is given by $$\mathcal P_b(B,C)=\mathcal P_c(\mathcal P_r(B,\mathbf{0}_{n\times m}),\mathcal P_r(\mathbf{0}_{q\times p},C)).$$
\end{definition} 
\begin{remark} We can see that previous definition allows to recover the formalism related with Pasting and Reversing of vectors, that is
$$\mathcal R_r(A)=\begin{pmatrix}\mathcal R(a_1)\\\vdots\\\mathcal R(a_n)\end{pmatrix},\quad\textrm{with}\quad a_i\in \mathbb K^m,$$ $$\mathcal R_c(A)=\begin{pmatrix}\mathcal R(a_1)&\cdots&\mathcal R(a_m)\end{pmatrix}\quad\textrm{with}\quad a_j\in \mathbb K^n,$$  $$\mathcal R(A)=\begin{pmatrix}\mathcal R(a_m)&\cdots&\mathcal R(a_1)\end{pmatrix}=\quad\textrm{with}\quad a_j\in \mathbb K^n,$$ 
$$\mathcal P_r(A,B)= \begin{pmatrix}
\mathcal P(a_1,b_1)\\\vdots\\\mathcal{P}(a_n,b_n)
\end{pmatrix},\quad\textrm{where}\quad a_i\in K^m,\,b_i\in K^p$$ and $$\mathcal P_c(A,C)=\begin{pmatrix}\mathcal P(a_1,c_1)&\ldots&\mathcal P(a_m,c_m)\end{pmatrix},\quad\textrm{where}\quad a_j\in K^n,\,c_j\in K^q.$$ 
\end{remark}
To illustrate this formalism we rewrite the following items.
$$\begin{array}{llll}
21.&\mathcal R^2_r(A)=A&
22.&\mathcal R^2_c(A)=A\\
23.&\mathcal R_r(\mathcal P_r(A,B))=\mathcal P_r(\mathcal R_r (B),\mathcal R_r (A))&
24.&\mathcal R_c(\mathcal P_c(A,B))=\mathcal P_c(\mathcal R_c (B),\mathcal R_c (A))\\
25.&\mathcal P_r(\mathcal P_r(A,B),C)=\mathcal P_r(A,\mathcal P_r(B,C))&
26.&\mathcal P_c(\mathcal P_c(A,B),C)=\mathcal P_c(A,\mathcal P_c(B,C))\\
27.&\mathcal R_r(\alpha A+\beta B)=\alpha\mathcal R_r(A)+\beta\mathcal R_r(B)&
28.&\mathcal R_c(\alpha A+\beta B)=\alpha\mathcal R_c(A)+\beta\mathcal R_c(B)\\
29.&\mathcal P_r(\mathcal M_{n\times m},\mathcal M_{n\times p})= \mathcal M_{n\times(m+p)}&
30.&\mathcal P_c(\mathcal M_{n\times m},\mathcal M_{l\times m})= \mathcal M_{(n+l)\times m}
\end{array}$$
$$\begin{array}{ll}
31.&\mathcal P_r(A,B)=A(\mathcal P_r(\mathcal P_c(I_n,\mathbf{0}_{(n-m)\times m}),\mathbf{0}_{n\times p})+\mathcal P_r (\mathbf{0}_{n\times m},B)\\
32.&\mathcal P_c(A,B)=A(\mathcal P_c(P_r(I_n,\mathbf{0}_{n\times (m-q)}), \mathbf{0}_{n\times p}))+\mathcal P_c(\mathbf{0}_{n\times m},B)\end{array}$$
$$\begin{array}{llll}
43.& (\mathcal R_r(A))^T=\mathcal R_c(A^T)&
44.& (\mathcal R_c(A))^T=\mathcal R_r(A^T)\\
45.&(\mathcal P_c(A,B))=\mathcal P_r(A^T,B^T)&
46.&(\mathcal P_r(A,B))=\mathcal P_c(A^T,B^T)\\
47.&\mathcal R_r(AB)=A\mathcal R_r(B)&
48.&\mathcal R_c(AB)=\mathcal R_c(A)B\\
49.&\det(\mathcal R_c(A))=(-1)^{\lfloor\frac{n}{2}\rfloor}\det A&
50.&\det(\mathcal R_r(A))=(-1)^{\lfloor\frac{n}{2}\rfloor}\det A\\
51.&(\mathcal R_c(A))^{-1}=\mathcal R_r(A^{-1})&
52.&(\mathcal R_r(A))^{-1}=\mathcal R_c(A^{-1})\end{array}$$
$$\begin{array}{ll}
81.&\mathcal R_r(M^{(k)})=M^{(n-k+1)}\mathcal R(I_{n-1}), 1\leq k\leq n\\
82.&\displaystyle \bigwedge_{i=1}^{n-1} \mathcal R_r({v_i})=(-1)^{\left\lceil\frac{3n}{2}\right\rceil}\mathcal R_r\left(\bigwedge_{i=1}^{n-1} (v_{i})\right)\end{array}$$
$$\begin{array}{llll}
83.&\mathcal R (A)=M_{\mathcal R_c}AM_{\mathcal R_r}&
84.&\mathcal R (\mathcal R (A))=A\\
85.&\mathcal R(\mathcal P(A,B))=\mathcal P(\mathcal R(B),\mathcal R(A))&
86.&\mathcal P( \mathcal P(A,B),C)=\mathcal P(A, \mathcal P(B, C))\end{array}$$
$$\begin{array}{ll}
87.&b,c\in K,\,p=n,q= m,\mathcal R (bA+cB)=b\mathcal R(A)+c\mathcal R(B)\end{array}$$
$$\begin{array}{llll}88.&\mathcal P (A, B)=\mathcal M_{r\times s}(K)&
97.&\mathcal R(I_n)=I_n\\
98.&\mathcal R(A)=\mathcal R_c\left(\mathcal R_r(A)\right)&
99.&\mathcal R(A)=\mathcal R_r\left(\mathcal R_c(A)\right)\\
100.&\mathcal R(AB)=\mathcal R(A)\mathcal R(B)&
101.&(\mathcal R(A))^{-1}=\mathcal R(A^{-1})\\
102.&\det (\mathcal R(A))=\det A&
103.&\mathrm{Tr}(\mathcal R(A))=\mathrm{Tr} A\\
104.&\mathcal R(A^T)=(\mathcal R(A))^T&
108.&\mathcal R(\mathcal P_b(A,B))=\mathcal P_B(\mathcal R(B),\mathcal R(A))\\
109.&\mathcal P_b(\mathcal P_b(A,B), C)=\mathcal P_b(A,\mathcal P_b (B, C))&
110.&\mathcal P_b (A, B)\in \mathcal M_{(n+p)\times (m+q)}\\
111.&\mathcal P_b(A, B)^T=\mathcal P_b(A^T, B^T)&
112.&\det (\mathcal P_b(A, B))=\det A  \det B\\
113.&\mathrm{Tr}(\mathcal P_b(A, B)=\mathrm{Tr} A+\mathrm{Tr} B&
114.&\mathcal P_b(A, B)^{-1}=\mathcal P_b(A^{-1},B^{-1}).
\end{array}$$
The following result is an extension of Proposition \ref{proplint} according with the previous definition.
\begin{proposition}\label{proplint2}
Consider $V=\mathcal M_{n\times m}$. The following statements hold.
\begin{enumerate}
\item  $\{\mathcal R_c,\mathcal R_r, \mathcal R\}\subset Aut(V)$.
\item The transformation matrices of $\mathcal{R}$, $\mathcal R_r$ and $\mathcal R_c,$ are given respectively by  \begin{displaymath} M_\mathcal {R}=(\delta_{i,nm-j+1})_{nm\times nm},\quad M_{\mathcal {R}_r}=(\delta_{i,m-j+1})_{m\times m},\quad M_{\mathcal {R}_c}=(\delta_{n-i+1,j})_{n\times n},\end{displaymath} where $M_\mathcal R$ acts over $A$ written as vector in $\mathbb{K}^{nm}$, $M_{\mathcal R_r}$ acts over $A$ multiplying it by right and $M_{\mathcal R_c}$ acts over $A$ multiplying it by left.
\item $\mathcal R(A)=M_{\mathcal R_c}AM_{\mathcal R_r}$.
\item $M_{\mathcal R_c}M_{\mathcal R_r}=I_n \Leftrightarrow n=m$.
\item $\ker(\mathcal R- id)^{\perp}=\ker(\mathcal R+id) $.
\item $\ker (\mathcal R-id)=\ker (\mathcal R_r-\mathcal R_c)$.
\item $\ker (\mathcal R+id)=\ker (\mathcal R_r+\mathcal R_c)$.
\item $V=\ker(\mathcal R-id)\oplus \ker(\mathcal R+id)=(\ker(\mathcal R_c-id)\cap \ker(\mathcal R_r-id))\oplus (\ker(\mathcal R_c+id)\cap \ker(\mathcal R_r+id))\oplus (\ker(\mathcal R_c+id)\cap \ker(\mathcal R_r-id))\oplus (\ker(\mathcal R_c-id)\cap \ker(\mathcal R_r+id))$.
\item $\dim \ker(\mathcal R-id)=\left\lceil\frac{nm}{2}\right\rceil$, $\dim \ker(\mathcal R+id)=\left\lfloor\frac{nm}{2}\right\rfloor$, $\dim (\ker(\mathcal R_c-id)\cap \ker(\mathcal R_r-id))=\left\lceil\frac{n}{2}\right\rceil\left\lceil\frac{m}{2}\right\rceil$, $\dim (\ker(\mathcal R_c+id)\cap \ker(\mathcal R_r+id))=\left\lfloor\frac{n}{2}\right\rfloor \left\lfloor\frac{m}{2}\right\rfloor$, $\dim (\ker(\mathcal R_c+id)\cap \ker(\mathcal R_r-id))=\left\lfloor\frac{n}{2}\right\rfloor\left\lceil\frac{m}{2}\right\rceil$, $\dim (\ker(\mathcal R_c-id)\cap \ker(\mathcal R_r+id))=\left\lfloor\frac{n}{2}\right\rfloor\left\lceil\frac{m}{2}\right\rceil$.
 \end{enumerate}
\end{proposition}

\begin{proof}
It follows from Proposition \ref{proplint} adapted for matrices and previous (above) items 97 to 100.
\end{proof}
\begin{remark} In \cite{acchro2} were defined and studied the subspaces associated with items 3, 4 and 5 of Proposition \ref{proplint2}. We see that for square matrices $M_{\mathcal R_c}=M_{\mathcal{R}_r}$, although they are acting in different ways (left for columns and right for rows). From now on for square matrices we write $M_\mathcal R$.

\end{remark}
The following result, comes from the previous definitions.
\begin{proposition}\label{lem12}
	The following statements hold.
	\begin{enumerate}
		\item Consider $A\in M_{n\times m}$ and $B\in M_{n\times r}$, then $(\mathcal P_r(A,B))^T=\mathcal P_c(A^T,B^T).$
		\item  Consider $A\in M_{n\times n}$, then $\mathcal R(Adj A)=Adj(\mathcal R(A))$.
		\item Consider $A'$ as the augmented matrix of $A$ with $b$, then $A'=\mathcal P_r(A,b)$.
	\end{enumerate}
\end{proposition}
\begin{proof}
We proceed according to each item.
\begin{enumerate}
	\item 	Assume $A\in M_{n\times m}$ and $B\in M_{n\times p}$ then  
	$$(\mathcal P_r(A,B))^T=(\mathcal P(a_i,b_i))^T\quad\textrm{where}\quad a_i\in K^m,\,b_i\in K^p.$$
	Now, by Theorem \ref{teor1} we have
	\begin{equation*}\begin{split}(\mathcal P(a_i,b_i))^T&=((S\circ\varphi)(a_i,b_i))^T \\
	&=(S(\varphi_1(a_i),\varphi_2(b_i)))^T\\
	&=(\varphi_1(a_i)+\varphi_2(b_i))^T. \end{split} 
	\end{equation*}
	Therefore, by properties of transpose matrices we arrive to 
	$$(\varphi_1(a_i)+\varphi_2(b_i))^T =\varphi_1(a_i)^T+\varphi_2(b_i)^T $$
	and for instance
	\begin{equation*}\begin{split}\varphi_1(a_i)^T+\varphi_2(b_i)^T &=S(\varphi_1(a_i)^T,\varphi_2(b_i)^T )\\
	(S\circ\varphi)(a_i^T,b_i^T)&=\mathcal P(a_i^T,b_i^T).\end{split}\end{equation*}
	Therefore 
	$$(\mathcal P_r(A,B))^T=\mathcal P_c(A^T,B^T).$$ 
	\item Due to $Adj A=|A|A^{-1},$ we have that $Adj(\mathcal R(A))=|\mathcal R(A)|(\mathcal R(A))^{-1}$. Due to property 102 we get $|\mathcal R(A)|(\mathcal R(A))^{-1}=|A|(\mathcal R(A))^{-1}$. Now, by property 101 $|A|(\mathcal R(A))^{-1}=|A|\mathcal R(A^{-1})=\mathcal R(|A|A^{-1}),$ thus $\mathcal R(Adj A)=Adj(\mathcal R(A))$
	\item We see that $$A'=\left[\begin{array}{@{}ccc|c@{}}a_{11}&\cdots&a_{1m}&b_1\\\vdots&\ddots&\vdots&\vdots\\a_{n1}&\cdots&a_{nm}&b_n\end{array}\right].$$ Moreover, $\varphi_1(A)\in M_{n\times(m+1)}$ and $\varphi_2(A)\in M_{n\times(m+1)}$. In this way $A'=\varphi_1(A)+\varphi_2(A)=S(\varphi_1(A),\varphi_2(A)),$ for instance $(S\circ\varphi)(A,b)=\mathcal P_r(A,b),$ así $A'=\mathcal P_r(A,b)$
\end{enumerate}
\end{proof}
\begin{remark} We note that item 2 is not true for reversing by rows and columns due to properties 49, 50, 51 and 52.
\end{remark}

Considering the vector space of $n\times m$ matrices, we say that a mapping is \textit{palindromicing by rows} (resp. by \textit{columns}) whether it transform any $n\times m$ matrix into a palindromic matrix by rows (resp. by columns), i.e., it is epimorphism from $\mathcal M_{n\times m}$ to $W_p^r(n\times m)$ (resp. $W_p^c(n\times m)$). In a similar way, we say that a mapping is \textit{antipalindromicing by rows} (resp. by \textit{columns}) whether it transform any $n\times m$ matrix into an antipalindromic matrix by rows (resp. by columns), i.e., it is epimorphism from $\mathcal M_{n\times m}$ to $W_a^r(n\times m)$ (resp. $W_a^c(n\times m)$). Moreover, we say that palindromicing and antipalindromicing mappings by rows (resp. by columns) are \textit{canonical},  denoted by $\mathcal F_p^r$ and $\mathcal F_a^r$ (resp. $\mathcal F_p^c$ and $\mathcal F_a^c$) respectively, whether they are linear mappings and for all $A\in \mathcal M_{n\times m}$ they satisfy $A=(\mathcal F_p^r+\mathcal F_a^r)(A)=(\mathcal F_p^c+\mathcal F_a^c)(A)$, $\mathcal R_r(A)=(\mathcal F_p^r-\mathcal F_a^r)(A)$ and $\mathcal{R}_c=(\mathcal F_p^c-\mathcal F_a^c)(A)$. From now on we only consider canonical palindromicing and antipalindromicing mappings by rows (resp. by columns), which will be called \textit{Palindromicing} and \textit{Antipalindromicing} mappings by rows (resp. by columns). Nevertheless, we can consider Palindromicing and Antipalindromicing mappings for matrices with respect $\mathcal{R}=\mathcal{R}_r\circ \mathcal{R}_c$, denoted as $\mathcal{F}_p$ and $\mathcal{F}_a$ respectively. The following proposition, which is an extension of Proposition \ref{proppalant} for matrices, give us the characterization of Palindromicing and Antipalindromicing mappings for matrices.
\begin{proposition}[Palindromicing \& Antipalindromicing mappings in $\mathcal{M}_{n\times m}$]\label{proppalantmat}
The following statements hold.
\begin{enumerate}
\item $\mathcal F_p^r$, $\mathcal F_a^r$, $\mathcal F_p^c$, $\mathcal F_a^c$, $\mathcal F_p$ and $\mathcal F_a$ are given by
\begin{displaymath}\begin{array}{llrllllrll}
 \mathcal F_p^r&: &\mathcal{M}_{n\times m}&\rightarrow&W_p^r(n\times m)&& \mathcal F_a^r&: \mathcal{M}_{n\times m}&\rightarrow&W_a^r(n\times m) \\ &&&&&,&&&&\\&& A&\mapsto&\frac12(A+\mathcal R_r(A))&&& A&\mapsto&\frac12(A-\mathcal R_r(A))\end{array}.
 \end{displaymath} 
 
\begin{displaymath}\begin{array}{llrllllrll}
 \mathcal F_p^c&: &\mathcal{M}_{n\times m}&\rightarrow&W_p^c(n\times m)&& \mathcal F_a^c&: \mathcal{M}_{n\times m}&\rightarrow&W_a^c(n\times m) \\ &&&&&,&&&&\\&& A&\mapsto&\frac12(A+\mathcal R_c(A))&&& A&\mapsto&\frac12(A-\mathcal R_c(A))\end{array}.
 \end{displaymath} 

\begin{displaymath}\begin{array}{llrllllrll}
 \mathcal F_p&: &\mathcal{M}_{n\times m}&\rightarrow&\mathbf{PA}(n\times m)&& \mathcal F_a&: \mathcal{M}_{n\times m}&\rightarrow&\mathbf{aPA}(n\times m) \\ &&&&&,&&&&\\&& A&\mapsto&\frac12(A+\mathcal R(A))&&& A&\mapsto&\frac12(A-\mathcal R(A))\end{array}.
 \end{displaymath} 

\item $\ker(\mathcal F_p^r)=\mathrm{Im}(\mathcal{F}_a^r)$, $\ker(\mathcal F_a^r)=\mathrm{Im}(\mathcal{F}_p^r)$, $\ker(\mathcal F_p^c)=\mathrm{Im}(\mathcal{F}_a^c)$, $\ker(\mathcal F_a^c)=\mathrm{Im}(\mathcal{F}_p^c)$, $\ker(\mathcal F_p)=\mathrm{Im}(\mathcal{F}_a)$, $\ker(\mathcal F_a)=\mathrm{Im}(\mathcal{F}_p)$.
\item $\mathcal F_p^r(A)=A(M_\mathcal{R}+I_n)$, $\mathcal F_a^r(A)=A(M_\mathcal{R}-I_n)$, $\mathcal F_p^c(A)=(M_\mathcal{R}+I_m)A$, $\mathcal F_a^c(A)=(M_\mathcal{R}-I_m)A$, $\mathcal F_p(A)=M_\mathcal RAM_\mathcal{R}+A$ and $\mathcal F_a(A)=M_\mathcal RAM_\mathcal{R}-A$.
\end{enumerate}
\end{proposition}
\begin{proof}
We can see any matrix $A\in \mathcal{M}_{n\times m}$ as an array of $n$ row vectors belonging to $\mathbb{K}^m$ or an array of $m$ column vectors belonging to $\mathbb{K}^n$ or a vector belonging to $\mathbb{K}^{nm}$. Thus, the results are obtained in virtue of Proposition \ref{proppalant}.
\end{proof}

From basic linear algebra we know that the main diagonal  and trace are linear mappings, that is $\mathrm{diag}(A+B)=\mathrm{diag}(A)+\mathrm{diag}(B)$ and $\mathrm{diag}(c\cdot A)=c\cdot \mathrm{diag}(A)$, as well $Tr(A+B)=Tr(A)+Tr(B)$ and $Tr(c\cdot A)=c\cdot Tr(A)$. Thus, we arrive to the following elementary result.

\begin{lemma}\label{l1} Let $D$ be a diagonal matrix and $A=\begin{pmatrix}a_{ij}\end{pmatrix}_{n\times n}$, the following statements hold.
\begin{enumerate}
\item $\mathcal{R}(\mathrm{diag}(A))=\mathrm{diag}({\mathcal{R}(A)})$
\item $Tr(A)=Tr(\mathcal{R}(A))$
\item $\mathcal{R}_{c}(D)=\mathcal{R}_{r}(D)\Leftrightarrow \mathrm{diag}(D)=\mathrm{diag}(\mathcal{R}(D))$
\item $\mathcal{R}(\mathrm{diag}(D))=\mathrm{diag}(D)\Leftrightarrow \mathcal{R}D=D$
\end{enumerate}
\end{lemma}
\begin{proof}We proceed according to each item.
\begin{enumerate}
\item By Proposition \ref{proplint2} we see that $\mathcal R(\mathrm{diag}(A))=M_{\mathcal R}\mathrm{diag}(A)M_\mathcal R$, for instance $\mathcal R(\mathrm{diag}(A))=\mathrm{diag}(M_{\mathcal R}AM_\mathcal R)=\mathrm{diag}(\mathcal R(A))$.
    \item It follows from previous item.
    \item Consider $$D=\begin{pmatrix}d_{11}&0&\cdots &0\\0&d_{22}&\cdots&0\\\vdots& \ddots&\vdots\\0&0&\cdots&d_{nn}\end{pmatrix}.$$ By definition of Reversing by columns and rows we have $$\mathcal{R}_c(D)=\begin{pmatrix}0&0&\cdots&0&d_{nn}\\\vdots&\vdots&\ddots&\vdots&\vdots
    \\d_{11}&0&\cdots&0&0\end{pmatrix}$$ and $$\mathcal{R}_r(D)=\begin{pmatrix}0&0&\cdots&0&d_{11}\\\vdots&\vdots&\ddots&\vdots&\vdots\\d_{nn}&0&\cdots&0&0\end{pmatrix}.$$ Owing to $\mathcal{R}_c(D)=\mathcal{R}_r(D)$ we have that $d_{kk}=d_{(n-k+1)(n-k+1)}$ then $\mathrm{diag}(D)=(d_{11},d _{22},\ldots,d_{22},d_{11})=\mathrm{diag}(\mathcal{R}(D))$. Conversely, consider $$D=\begin{pmatrix}d_{11}&0&\cdots &0\\0&d_{22}&\cdots&0\\\vdots& \ddots&\vdots\\0&0&\cdots&d_{nn}\end{pmatrix},$$ due to $\mathrm{diag}(D)=\mathrm{diag}(\mathcal{R}(D))$, then $\mathrm{diag}(D)=(d_{11},d_{22},\ldots,d_{22},d_{11})$. Now, applying $\mathcal{R}_c\,\,\textrm{and}\,\,\mathcal{R}_r$ over $D$, we have
     $$\mathcal{R}_c(D)=\begin{pmatrix}0&0&\cdots&0&d_{11}\\\vdots&\vdots&\ddots&\vdots&\vdots\\
     d_{11}&0&\cdots&0&0\end{pmatrix}=\begin{pmatrix}0&0&\cdots&0&d_{11}\\\vdots&\vdots&\ddots&\vdots&\vdots
     \\d_{11}&0&\cdots&0&0\end{pmatrix}=\mathcal{R}_rD.$$
     \item By item $1$ we have that $\mathcal{R}(\mathrm{diag}(D))=\mathrm{diag}(\mathcal{R}(D))$ and by hypothesis we have that $\mathrm{diag}(\mathcal{R}(D))=\mathrm{diag}(D)$. Now, by item $3$ we obtain that $\mathcal{R}_c(D)=\mathcal{R}_r(D)$ and applying  $\mathcal{R}_c$ we get $\mathcal{R}_c^2(D)=\mathcal{R}(D)$. Therefore $D=\mathcal{R}(D)$. Conversely,
due to $\mathcal{R}=\mathcal{R}_c\circ\mathcal{R}_r$ then $\mathcal{R}_c(\mathcal{R}_r(D))=D$. Applying $\mathcal{R}_c$ we have $\mathcal{R}_r(D)=\mathcal{R}_c(D)$ and items $3$ and $1$ lead us to $\mathrm{diag}(D)=\mathrm{diag}(\mathcal{R}(D))=\mathcal{R}(\mathrm{diag}(D))$.
\end{enumerate}
\end{proof}
\begin{remark} In item $3$ of Lemma \ref{l1} we can see that the main diagonal of $D$ is palindromic, while in item $4$ we observe an equivalence between the palindromic of the diagonal matrix with the palindromic of the main diagonal of the matrix.
\end{remark}

In this way we obtain the following results.

\begin{theorem}\label{t1} Let $p$, $q$, $r$ and $s$ be the characteristic polynomials of $n\times n$ matrices  $A$, $\mathcal R(A)$, $\mathcal{R}_c(A)$ and $\mathcal R_r(A)$ respectively. The following statements hold.
\begin{enumerate}
\item $\mathrm{diag}(\mathcal{R}(A)-\lambda I)=\mathcal{R}(\mathrm{diag}(A-\lambda I))$
\item $p=q$
\item $r=s$
\item $p(A)=p(\mathcal{R}(A))$
\item $ p(\mathcal{R}_r (A))=p(\mathcal{R}_c (A))$
\end{enumerate}
\end{theorem}
\begin{proof} We proceed according to each item.
\begin{enumerate}
\item Due to $\mathrm{diag}$ is a linear mapping and $\lambda I$ is a palindromic matrix, we obtain $\mathrm{diag}(\mathcal{R}(A)-\lambda I)=\mathrm{diag}(\mathcal{R}(A))-\mathrm{diag}(\mathcal{R}(\lambda I))$. By item $1$ in Lemma \ref{l1}, we conclude $\mathrm{diag}(\mathcal{R}(A)-\lambda I)=\mathcal{R}(\mathrm{diag}(A-\lambda I))$.
\item Consider $A=(a_{ij})_{n\times n}$, thus $A-\lambda I=(b_{ij})$, where
$$b_{ij}=\left\{
\begin{array}{lr}
 a_{ij}-\lambda,& i=j\\
 a_{ij},&i\neq j
\end{array}
\right.$$
Now $\mathcal{R}(A)-\lambda I=(c_{ij})$ where
$$c_{ij}=\left\{
\begin{array}{lr}
 a_{(n-i+1)(n-j+1)}-\lambda,& n-i+1=n-j+1\\
 a_{(n-i+1)(n-j+1)},&n-i+1\neq n-j+1
\end{array}
\right.$$
This lead us to
$$c_{ij}=\left\{
\begin{array}{lr}
 a_{(n-i+1)(n-j+1)}-\lambda,& i=j\\
 a_{(n-i+1)(n-j+1)},&i\neq j
\end{array}
\right.$$
Due to $\mathcal{R}(A-\lambda I)=\mathcal{R}(A)-\lambda I$,  in virtue of item $102$ in Section $2$, we have
$\det(\mathcal{R}(A-\lambda I))=\det(\mathcal{R}(A)-\lambda I))$, which implies that $p=q$.

\item Owing to $\mathcal R_c^2=I$ we can write $\det(\mathcal R_c A-\lambda I)$ as  $$\det(\mathcal R_c A-\lambda\mathcal R^2_c I)=\det(\mathcal R_c(A-\lambda\mathcal R_cI))=(-1)^{\lfloor\frac{n}{2}\rfloor}\det(A-\lambda\mathcal R_c I).$$ Due to $\mathcal R_cI=\mathcal R_rI$ and $[\delta_{i,n-j+1}]_{n\times n}\cdot I=I\cdot[\delta_{i,n-j+1}]_{n\times n}$, we obtain $$(-1)^{\lfloor\frac{n}{2}\rfloor}\det(A-\lambda\mathcal R_c I)=(-1)^{\lfloor\frac{n}{2}\rfloor}\det(A-\lambda\mathcal R_r I)=\det(\mathcal R_rA-\lambda I),$$ which implies  $r=s$.
\end{enumerate}
Items 4 and 5 correspond to the application of Cayley - Hamilton Theorem for Reversing, followed by items 1 and 2.
\end{proof}
\begin{remark} Analysis of Reversing to characteristic polynomials can be done exactly as in \cite{acchro1} and \cite{MaRa}, in where were also studied palindromic and antipalindromic polynomials. The following results are were obtained as application of previous theorems and Proposition \ref{proplint}.
\end{remark}
\begin{theorem}[Reversing Jordan Form]\label{jordant} Assume $V=\mathbb K^n$, $\mathcal B_p$ and $\mathcal B_a$ canonical basis of palindromic and antipalindromic vectors of $V$. Let $\mathfrak{M}_p$ and $\mathfrak{M}_a$ be matrices formed by the pasting by rows (resp. by columns) of palindromic and antipalindromic vectors by column (resp. by rows) respectively. Then, up to isomorphisms, 
 Jordan form and similarity matrix of $M_{\mathcal R}$ are  given by $$J_\mathcal R=
\mathcal P_b(I_{\lceil \frac{n}{2}\rceil},-I_{\lfloor\frac{n}{2}\rfloor} \textit{ and }P_\mathcal R=
\mathcal P_r(\mathfrak M_p,\mathfrak M_a) \textit{ respectively}.$$ Furthermore $P_\mathcal R$ is a  symmetric matrix.
\end{theorem}
\begin{proof} Assume $v\in V$ as column vectors. In virtue of Proposition \ref{proplint} we obtain $J_\mathcal R$. Now, due to the eigenvalues $1$ and $-1$ correspond to palindromic and antipalindromic eigenvectors respectively, we can choose those belonging to $\mathcal B_p$ and $\mathcal B_a$ respectively, that is $v_{p_i}\in \mathcal B_p$ and $v_{a_i}\in \mathcal B_a$. Thus $\mathfrak{M}_p=\mathcal P_r(v_{p_1},\ldots,v_{p_{\lceil\frac{n}{2}\rceil}})$ and $\mathfrak{M}_p=\mathcal P_r(v_{a_1},\ldots,v_{a_{\lfloor\frac{n}{2}\rfloor}})$. Finally, $P_\mathcal R$ is obtained as $P_\mathcal R=\mathcal P(\mathfrak{M}_p,\mathfrak{M}_a)$, which is symmetric. Under assumption of $v\in V$ as row vectors, the proof is similar. 
\end{proof}
\begin{theorem}\label{t2} Suppose that the $n\times n$ matrix $A$ admits a Jordan form. The following statements hold.
\begin{enumerate}
\item Jordan form is preserved under Reversing, where similarity matrix associated to Reversing of $A$ is Reversing by columns of similarity matrix associated to $A$. 
\item If the Jordan form is palindromic (resp. antipalindromic) and the similarity matrix associated to $A$ is palindromic (resp. antipalindromic), then $A$ is palindromic (resp. antipalindromic).
\end{enumerate}
\end{theorem}
\begin{proof} Due to $A$ admits a Jordan form, there exists a similarity matrix $P$,  i.e.,  $A=PJP^{-1}$. Now, we proceed according to each item.
\begin{enumerate}
\item By item 3 of Proposition \ref{proplint2} we have that $\mathcal R (A)=M_{\mathcal R_c}AM_{\mathcal R_r}$, that is $\mathcal R (A)=M_{\mathcal R_c}PJP^{-1}M_{\mathcal R_r}={\mathcal R_c}(P)J{\mathcal R_r(P^{-1})}$. Now, by items 98 to 101 we get $\mathcal R (A)=QJQ^{-1}$, where $Q=\mathcal R_c(P)$.
\item By hypothesis, item 100 and item 101 we have that $\mathcal R(A)=\mathcal R(PJP^{-1})=\mathcal R(P)\mathcal R(J)(\mathcal R (P))^{-1} =PJP^{-1}=A$. Similarly, assuming $\mathcal R(J)=-J$ and $\mathcal R(P)=-P$, we have that $R(A)=\mathcal R(PJP^{-1})=\mathcal R(P)\mathcal R(J)(\mathcal R (P))^{-1} =-PJP^{-1}=-A$.
\end{enumerate}
\end{proof}
\begin{remark} In general, Reversing of a Jordan form is not a Jordan form. Moreover, the converse of item 2 in general is not true.
\end{remark}
To illustrate the previous remark, we present the following example.
\begin{example}
	Consider $A=\begin{pmatrix}1&2&1\\0&-1&0\\-1&1&3\end{pmatrix}$, its Jordan matrix is $J=\begin{pmatrix}-1&0&0\\0&2&1\\0&0&2\end{pmatrix}$. Reversing of $J$ is given by $\mathcal R(J)=\begin{pmatrix}2&0&0\\1&2&0\\0&0&-1\end{pmatrix},$ which is not a Jordan matrix.
\end{example}

\begin{theorem} Consider $f\in \mathbb{R}[x]$ and the $n\times n$ matrices $A$ and $\mathcal{R}_r(A)$  with eigenvalues $\lambda_i$ and $\widetilde{\lambda}_i$ respectively, where $1\leq i\leq m\leq n$. The following statements hold.

\begin{enumerate}
\item $f(\lambda_i)$ is eigenvalue of $f(\mathcal R(A))$
\item $f(\widetilde{\lambda_i})$ is eigenvalue of $f(\mathcal R_c(A))$
\end{enumerate}
\end{theorem}
\begin{proof} For basic linear algebra we know that $f(\lambda_i)$, $1\leq i\leq m\leq n$, are the eigenvalues of $f(A)$. Now we proceed according to each item.
\begin{enumerate}
\item  By Theorem \ref{t1} we have that eigenvalues of $A$ are the same of $\mathcal R(A)$. Thus, $f(\lambda_i)$ is eigenvalue of $f(\mathcal R(A))$ for $1\leq i\leq m\leq n$.
\item By Theorem \ref{t1} we have that eigenvalues of $\mathcal R_rA$ are the same of $\mathcal R_c(A)$. Thus, $f(\widetilde{\lambda_i})$ is eigenvalue of $f(\mathcal R_c(A))$ for $1\leq i\leq m\leq n$.
\end{enumerate}
The proof is done.
\end{proof}
\begin{theorem} Assume $f: \mathfrak A\subset \mathbb R\rightarrow \mathfrak B\subset\mathbb{R}$ is analytic and let $A$ be an $n\times n$ matrix. The following statements hold.

\begin{enumerate}
\item $f(\mathcal R_r(A))=\mathcal R(f(\mathcal R_c(A)))$
\item $f(\mathcal R_c(A))=\mathcal R(f(\mathcal R_r(A)))$
\item $f(\mathcal R(A))=\mathcal R(f (A))$
\end{enumerate}
\end{theorem}
\begin{proof} By items 21 and 22, we see that $M_\mathcal{R}^2=I_n$ and due to $f$ is analytic, we have that \begin{displaymath}
 f(A)=\sum_{k=0}^\infty a_kA^k.
\end{displaymath} Now, we proceed according to each item.

\begin{enumerate}
\item We see that $(\mathcal R_r(A))^k=(AM_{\mathcal{R}})^k=M_{\mathcal{R}}M_{\mathcal{R}}AM_{\mathcal{R}}AM_{\mathcal{R}}\cdots AM_{\mathcal{R}}AM_{\mathcal{R}}=M_{\mathcal{R}}BM_{\mathcal{R}}=\mathcal{R}(B)$, where $B=M_{\mathcal{R}}AM_{\mathcal{R}}A\cdots M_{\mathcal{R}}AM_{\mathcal{R}}A=(\mathcal R_c(A))^k$, therefore $f(\mathcal R_r(A))=\mathcal R(f(\mathcal R_c(A)))$.
\item We see that $(\mathcal R_c(A))^k=(M_{\mathcal{R}}A)^k=M_{\mathcal{R}}AM_{\mathcal{R}}A\cdots M_{\mathcal{R}}AM_{\mathcal{R}}AM_{\mathcal{R}}M_{\mathcal{R}}=M_{\mathcal{R}}BM_{\mathcal{R}}=\mathcal{R}(B)$, where $B=AM_{\mathcal{R}}AM_{\mathcal{R}}\cdots AM_{\mathcal{R}}AM_{\mathcal{R}}=(\mathcal R_r(A))^k$, therefore $f(\mathcal R_c(A))=\mathcal R(f(\mathcal R_r(A)))$.
\item By item 100, taking $A=B$, and proceeding inductively we obtain $(\mathcal R(A))^k=\mathcal R(A^k)$, therefore $f(\mathcal R(A))=\mathcal R(f (A))$.
\end{enumerate}
\end{proof}

\section*{Final Remarks}
In this paper we extend some results presented in \cite{acarnu,acchro2}, following their same philosophy, relating Pasting and Reversing with vector spaces and matrix theory. We presented new results involving Pasting and Reversing over matrices, including properties of Jordan forms, eigenvalues, eigenvectors, analytic functions the introduction of transformations to obtain palindromic and antipalindromic vectors and matrices from any vector or any matrix. Nevertheless, there are a plenty of questions in matrix theory related with Pasting and Reversing that cannot be considered here, we wish that this paper can motivate to the readers to develop such questions.\\

Currently, there are some research projects involving Pasting and Reversing over other mathematical structures, some of them by author different to the author of this paper. Some of such projects include applications to orthogonal polynomials, differential equations, difference equations, quantum mechanics, topology, group theory, algebraic varieties, Ore rings, combinatorial dynamics, numerical analysis, graph theory, coding theory, statistics, among others.

\section*{Acknowledgements}
The publication of this paper is totally supported by Vicerrector\'{\i}a de Investigaciones - Universidad Sim\'on Bol\'{\i}var (Barranquilla - Colombia). The authors thank to Paola Amar by the sponsoring of this research, as well to David Bl\'azquez-Sanz and \'Angela Mariette Rodr\'{\i}guez by their useful comments and suggestions on this work. The first author acknowledges to Ecos Nord by the stays of research in where this work was improved. The second author acknowledges to Colciencias by the grant J\'ovenes Investigadores, which allow her to study the master program in mathematics at Universidad del Norte under the supervision of the first author. Finally, we thank to Greisy Morillo by their hospitality and support during the final part of this work.


\begin{thebibliography}{18}
\newcommand{\ar}[1]{``#1''}
\newcommand{\lb}[1]{{\sl``#1''\/}}
\newcommand{\co}[1]{{\sl#1\/}}
\newcommand{\rv}[1]{{\it#1\/}}
\newcommand{\vl}[1]{{\bf#1\/}}
\newcommand{\vol}[1]{vol.~\vl{#1}}
\newcommand{\nnu}[1]{\mbox{n.~#1}}
\newcommand{\pg}[1]{\mbox{#1}}


\bibitem{Ac1}P.B. Acosta Hum\'anez,{\it Genealogy of simple permutations with order a power of two} (Spanish). Revista Colombiana de Matem\'aticas, {\bf 42}, (2008) 1--14.

\bibitem{Ac2} P.B. Acosta-Hum\'anez,  {\it Pasting operation and the square of natural numbers} (Spanish),
Civilizar, {\bf 4}, (2003) 85--97.

\bibitem{acarnu} P.B. Acosta-Hum\'anez, M. Aranda, R. N\'unez, {\it Some Remarks on a Generalized Vector Product}. Integraci\'on, \textbf{29}, (2011) 151--162

\bibitem{acchro1} P. Acosta-Hum\'anez, A. Chuquen \& A. Rodr\'{\i}guez, \emph{Pasting and Reversing operations over some rings}, Bolet\'{\i}n de Matem\'aticas, \textbf{17}, (2010) 143--164

\bibitem{acchro2} P. Acosta-Hum\'anez, A. Chuquen, \emph{Pasting and Reversing operations over some vector spaces}, Bolet\'{\i}n de Matem\'aticas,  \textbf{20}, (2013), 145--161

\bibitem{AM}P.B. Acosta-Hum\'anez, O.E. Mart\'inez, {\it Simple permutations with order $4n + 2$}. Preprint arXiv:1012.2076v1.

\bibitem{AM1}P.B. Acosta-Hum\'anez, O.E. Mart\'inez, {\it Simple permutations with order $4n + 2$ by means of Pasting and Reversing}, Qualitative Theory of Dynamical Systems, \textbf{14}, (2015), 1--30

\bibitem{AM2}P.B. Acosta-Hum\'anez, O.E. Mart\'inez, {\it Simple permutations with order $4n + 2$. Part I}. Preprint arXiv:1012.2076v2.

\bibitem{acmoro} P.B. Acosta-Hum\'anez, P. Molano \& A. M. Rodr\'{\i}guez, \emph{Some Remarks on Pasting and Reversing in Natural Numbers} (Spanish), Matua, \textbf{2}, (2015), 65--90.

\bibitem{Br} R. Brualdi \& H. Ryser, {\it Combinatorial Matrix Theory}, Cambridge University Press, 1991.

\bibitem{Ev} H. Eves, {\it Elementary Matrix Theory}, Dover, 1980.

\bibitem{Ga} F. Gantmacher, {\it The Theory of Matrices}, American Mathematical Society, Vol. 1., 2000.

\bibitem{Ge} P. Gerdes, {\it Adventures in the World of matrices}, Nova Science Publishers, 2007.

\bibitem{hall} M. Hall, {\it Combinatorial Theory}, Wiley \& Sons, Second Edition, New York, 1983.

\bibitem{lantis} P. Lancaster \& M. Tismenetsky {\it The Theory of Matrices}, Computer science and applied mathematics, Academic Press, Second Edition, 1985.

\bibitem{Le} D. Lewis, {\it Matrix Theory}, World Scientific, 1991.

\bibitem{MaRa} I. Markovsky \& S. Rao, \emph{Palindromic polynomials, time-reversible systems, and conserved quantities}, in ``Proceedings 16th Mediterranean Conference on Control and Automation, Congress Centre, Ajaccio, France, June 25--27, 2008, IEEE (2008), 125--130.

\bibitem{We} J. Wedderburn, {\it Lectures on Matrices}, Colloquium Publications, 1980.

\end{thebibliography}
\end{document}